\documentclass[12pt]{amsart}

\usepackage{enumerate}

\usepackage[active]{srcltx}
\usepackage{nicefrac}

\usepackage{tikz}
\usepackage{graphicx}
\usetikzlibrary{shapes,patterns,calc,snakes}

\usepackage{amsthm,amssymb,setspace}
\usepackage[pagebackref,colorlinks,linkcolor=red,citecolor=blue,urlcolor=blue,hypertexnames=true]{hyperref}
\usepackage{amsrefs}
\usepackage[matrix, arrow]{xy}

\DeclareMathOperator\ord{ord}

\renewcommand{\deg}{{\rm deg}}

\newcommand{\N}{\mathbb N}
\newcommand{\R}{\mathbb R}
\newcommand{\Q}{\mathbb Q}
\newcommand{\C}{\mathbb C}
\newcommand{\Z}{\mathbb Z}

\newcommand{\B}{\mathcal B}
\newcommand{\cc}{\ensuremath{\mathbb{C}}}

\newcommand{\pp}{\ensuremath{\mathbb{P}}}
\newcommand{\rr}{\ensuremath{\mathbb{R}}}
\newcommand{\zz}{\ensuremath{\mathbb{Z}}}

\newcommand{\scrB}{\mathcal B}
\newcommand{\dsum}{\bigoplus}
\newcommand{\lot}{\text{l.o.t.}}

\theoremstyle{plain}
\newtheorem*{theorem*}{Theorem}
\newtheorem{theorem}{Theorem}[section]

\newtheorem{lemma}[theorem]{Lemma}
\newtheorem{proposition}[theorem]{Proposition}

\newtheorem{proclaim}{Claim}[theorem]
\newtheorem{defn}[theorem]{Definition}
\newtheorem{rem}[theorem]{Remark}
\newtheorem{cor}[theorem]{Corollary}
\newtheorem{prop}[theorem]{Proposition}

\theoremstyle{definition}
\newtheorem*{definition*}{Definition}
\newtheorem{definition}[theorem]{Definition}
\newtheorem{question}[theorem]{Question}

\theoremstyle{remark}

\newtheorem{example}[theorem]{Example}

\begin{document}

\onehalfspace

\title{How fast do polynomials grow on semialgebraic sets?}

\author{Pinaki Mondal}
\address{Pinaki Mondal, Weizmann Institute of Science, Israel}
\email{pinaki@math.toronto.edu}

\author{Tim Netzer}
\address{Tim Netzer, Universit\"at Leipzig, Germany}
\email{netzer@math.uni-leipzig.de}

\begin{abstract}
We study the growth of polynomials on semialgebraic sets. For this purpose we associate a graded algebra to the set, and address all kinds of questions about finite generation. We show that for  a certain class of sets, the algebra is finitely generated. This implies that the total degree of a polynomial determines its growth on the set, at least modulo bounded polynomials. We however also provide several counterexamples, where there is no connection between total degree and growth. In the plane, we give a complete answer to our questions for certain simple sets, and we  provide a systematic construction for examples and counterexamples.  Some of our counterexamples are of particular interest for the study of moment problems, since none of the existing methods seems to be able to decide the problem there. We finally also provide new three-dimensional sets, for which the algebra of bounded polynomials is not finitely generated.
\end{abstract}

\maketitle

\tableofcontents

\section{Introduction} Let $\R[x]$ be the polynomial algebra in $n$ variables $x=(x_1,\ldots,x_n)$. For $d\in \N$ we denote by $\R[x]_d=\left\{ p\in\R[x]\mid \deg(p)\leq d\right\}$ the finite-dimensional subspace of polynomials of total degree at most $d.$ For a set $S\subseteq \R^n$ let $$\B_d(S):=\left\{ p\in \R[x] \mid  p^2 \leq q\ \mbox{Êon } S, \mbox{ for some } q\in \R[x]_{2d}\right\}$$ denote the set of  those polynomials, that grow on $S$ as if they were of degree at most $d.$ Cleary $\R[x]_d\subseteq \B_d(S)$ for all $d$, and $\B_d(S)$ is closed under addition, which follows for example from the inequality $$(p+p')^2\leq (p+p')^2+(p-p')^2=2p^2+2p'^2.$$ $\B_0(S)$ is the algebra of {\it bounded polynomials} on $S$, and each $\B_d(S)$ carries the structure of a $\B_0(S)$-module. More general, $\B_d(S)\cdot \B_{d'}(S)\subseteq \B_{d+d'}(S),$ so we have a graded algebra  $$\B(S):=\bigoplus_{d\geq 0} \B_d(S).$$ $\B(S)$ can be identified with a subalgebra of $\R[x,t],$ wher
 e $t$ is a single new variable, by identifying $p\in \B_d(S)$ with $p\cdot t^d.$ Also note that $$\B_d(S_1\cup S_2)=\B_d(S_1)\cap \B_d(S_2)\ \mbox{Êand }\  \B(S_1\cup S_2)=\B(S_1)\cap \B(S_2)$$ holds for $S_1,S_2\subseteq \R^n.$ This follows from the fact that $q$ in the definition of $\B_d(S)$ can always assumed to be globally nonnegative. In fact you can always take some $C+D\Vert x\Vert^{2d}.$ Finally note that $\B_d(S)$ and thus $\B(S)$ do only depend on the behaviour of $S$ at infinity; if $S$ is changed inside of a compact set, no changes in $\B_d(S)$ and $\B(S)$ occur.
In this paper, we consider the following questions:

\begin{question}\label{one} Is $\mathcal B(S)$ a finitely generated algebra?\end{question}

\begin{question}\label{three} Is $\B_0(S)$ a finitely generated algebra? \end{question}
\begin{question}\label{two} Is every $\B_d(S)$ a finitely generated $\B_0(S)$-module?\end{question}
\begin{question}\label{four} If $\B_0(S)=\R,$ is every $\B_d(S)$ a finite-dimensional vector space? 
\end{question}

Note that a positive answer to Question \ref{one} yields a positive answer to all the other questions. Also note that  Question \ref{four}\ Êis just a special case of Question \ref{two}.  Let us start with some examples.
\begin{example}\label{ex1}
(1) Assume $S\subseteq \R^n$ contains a full-dimensional convex cone $K$ (e.g. $S=\R^n$ or $S=[0,\infty)^n$). For any $0\neq p\in\R[x]$ there is a point $0\neq a\in K$, on which the highest degree part of $p$ does not vanish. So on the half-ray through $a$, the polynomial $p^2$ cannot be bounded by a polynomial of degree smaller than $2\cdot \deg(p).$ This proves $\B_d(S)=\R[x]_d,$ and the answer to all questions is positive. Note that $\B(S)$ is generated by $t,x_1t,\ldots,x_nt.$

\smallskip\noindent
(2) Let $S=\left([0,1]\times \R\right) \cup \left(\R\times [0,1]\right)\subseteq\R^2$ be the union of a vertical and a horizontal strip. A polynomial $p$ belongs to $\B_d(S)$ if and only the degree of $p$ as both  a polynomial in $x_1$ and in $x_2$ is at most $d$. From this it is easy to see that the answer to all above questions is positive. Note that $\B_d(S)=\R[x]_d$ is not true here, since for example $x_1x_2\in \B_1(S),$ since $x_1^2x_2^2\leq x_1^2+x_2^2$ on $S$. However, $\B(S)$ is generated by $t,x_1t,x_1x_2t$ and $x_2t$.

\smallskip\noindent
(3) If $S$ is bounded, then $\B_0(S)=\B_d(S)=\R[x]$ for all $d$, and the answer to Question \ref{one} is obviously yes. In fact $\mathcal B(S)=\R[x,t]$ here.

\smallskip\noindent
(4) The set $S$ should be semialgebraic, if a positive answer to the above questions is to be expected. Indeed, consider $S=\left\{ (a,b)\in\R^2\mid 0\leq a, \exp(a)\leq b \right\}.$ Then for $p_d:=\sum_{i=0}^d \frac{1}{i!}x_1^i$  and $(a,b)\in S$ we have $$0\leq  p_d(a,b) \leq \exp(a) \leq q(a,b),$$ with $q=x_2$. So $p_d\in \B_1(S)$ for all $d\in\N$. On the other hand, $\B_0(S)=\R$ is easily checked, so  Question \ref{four} has a negative answer.

(5) Even in the semialgebraic case, the answer to Question \ref{four} can be negative. This example is Example 4.2 from \cite{cimane}. Consider $S=\left\{ (a,b)\in \R^2 \mid b^2(1-a^2)\geq 0\right\}$. So $S$ is the union of a vertical strip  and the $x_1$-axis. One checks that $\B_0(S)=\R$ holds.  On the other hand, $(x_1^dx_2)^2\leq x_2^2$ on $S$, so $x_1^dx_2\in \B_1(S)$ for all $d.$ So $\B_1(S)$ is not finite-dimensional, and Question \ref{four} has a negative answer. In particular, $\B(S)$ is not finitely generated. This example is however somewhat pathological, since $S$ has a lower-dimensional part, i.e. is not regular.

(6) Another pathology arising from non-regular sets is the following. Let $S$ be the $x_1$-axis in $\R^2$ alone. Then $\B_0(S)$ is not a finitely generated algebra, as one easily checks. So also the answer to Question \ref{one}Ê\ is negative. On the other hand, each $\B_d(S)$ is a finitely generated $\B_0(S)$-module, generated by $1,x_1,\ldots,x_1^d.$ 
\end{example}

We see that we should restrict to regular semialgebraic sets from now on, i.e. sets containing a dense open subset.
Let us recall what is actually known concerning the above questions. It seems like Question \ref{one} has not been explicitly studied for semialgebraic sets yet. The paper \cite{plsc} deals with Question \ref{three} and shows that the answer is yes in the two-dimensional regular case, and false in general for higher dimensions. The paper \cite{kr} provides an explicit three-dimensional such counterexample, based on a counterexample to Hilbert's 14th Problem of Nagata.
In the context of {\it moment problems}, Question \ref{four} has been extensively studied.  The works  \cite{kuma}, \cite{ma}, \cite{ne},  \cite{posc},   \cite{vi} all give  positive answers,  for large classes of sets.  To avoid confusion, we note that the question that is anwered in these papers is in fact the following: 

\begin{question}\label{five}
If $\B_0(S)=\R$, does there exist a function $\psi\colon\N\rightarrow\N$, such that for any two polynomials $p,q\in \R[x]$ with $p,q\geq 0 $ on $S$, one has $$\deg(p)\leq \psi\left(\deg(p+q)\right)?$$
\end{question}

\begin{lemma} \label{four=five}
For any set $S\subseteq \R^n$, Question \ref{four} and Question \ref{five} are equivalent. 
\end{lemma}
\begin{proof}
First assume that each $\B_d(S)$ is a finite dimensional vector space, and let $\psi(d)$ be the maximal degree of a polynomial in $\B_d(S)$. If $p,q$ are nonnegative on $S$, then $p\leq p+q$ on $S,$ so $p\in \B_{\deg(p+q)}(S)$.  This shows $\deg(p)\leq \psi\left( \deg(p+q)\right).$ Let conversely such a function $\psi$ be given. Obviously $\psi$ can be assumed to be non-decreasing. If $p^2\leq q$ for some $q\in \R[x]_{2d},$ then $q-p^2$ and $p^2$ are nonnegative on $S$.
We obtain $\deg(p^2) \leq \psi\left( \deg(p^2 +(q-p^2))\right)\leq \psi(2d).$ This shows that $\B_d(S)$ is finite dimensional.\end{proof}

Let us briefly recall some facts about the moment problem. Given a linear functional $\varphi\colon\R[x]\rightarrow \R$, one wants to determine whether it has a representing measure $\mu$, i.e. whether $$\varphi(p)=\int_{\R^n}p\ d\mu$$ holds for all $p\in\R[x].$ Also the support of $\mu$ is of interest here. A result by Haviland \cite{havi} states that $\varphi$ has a representing measure supported on a set $S\subseteq \R^n$ if and only if $\varphi(p)\geq 0$ holds for all polynomials $p$ which are nonnegative as functions on $S$. Unfortunately, describing all nonnegative polynomials on $S$ is a  hard problem. So Haviland's theorem becomes particularly helpful  if the nonnegativity condition can be weakend. Towards this goal one resctricts to {\it basic closed semialgebraic sets} $$S=\{a\in\R^n\mid p_1(a)\geq 0,\ldots,p_r(a)\geq 0\},$$ where $p_1,\ldots,p_r\in\R[x].$ The polynomials which are obviously nonnegative on $S$ are of the form $q^2$ and $q^2p_i$ for some $q\in\R[x].$
  
We say that $p_1,\ldots,p_r$ {\it solve the moment problem for $S$} if the conditions $$  \varphi(q^2p_i)\geq 0 \qquad \forall q\in \R[x], \ i=0,\ldots,r$$ (where $p_0=1$) are enough to ensure the existence of a representing measure for $\varphi$ on $S$. Such a weakened positivity condition on $\varphi$  can then for example be checked by a series of semidefinite programs (see for example \cite{ma}). The celebrated result of Schm\"udgen \cite{sch1} implies that if $S$ is bounded, then the finitely many products $p_1^{e_1}\cdots p_r^{e_r}$ ($e_i\in\{0,1\}$) always solve the moment problem for $S$. Another result of Schm\"udgen \cite{sch2} (see also \cite{ma,ne2}) provides a method to reduce the dimension in the moment problem. Given a nontrivial bounded polynomial $p\in\B_0(S)$, it is enough to check the moment problem on all fibres $S\cap \{p=\lambda\}$ of $p$ in $S$. Since the problem is usually easier in lower dimensions, this is very helpful. Now the significance of Questi
 on \ref{five} for the moment problem is the following. If $\B_0(S)= \R$  and Question \ref{five} has a positive answer, then the moment problem is not solvable, at least in dimension $\geq 2$, by a result of Scheiderer \cite{sc}. So it doesn't matter that the reduction result of Schm\"udgen cannot be applied, since the problem is not solvable anyway.  Since Question \ref{five} has a positive answer in so many cases, it has been asked whether the answer is always yes, for {\it regular} semialgebraic sets, i.e. sets containing a dense open subset (see for example \cite{pl}).  This would mean there is no gap between the results of Schm\"udgen and Scheiderer. In this paper we show, among other things, that this is false. In particular, deciding the moment problem for our counterexamples seems to call for completely new methods. 

Our contribution is the following. In Section \ref{tent} we show that Question \ref{one} has a positive answer for a large class of sets in arbitrary dimension, built of so-called  standard tentacles. 
In Section \ref{count} we provide a first regular two-dimensional example, for which Question \ref{four} (and thus any other of the questions as well) has a negative answer. We give a completely elementary and constructive proof. In Section \ref{plane} we use more elaborate techniques to examine planar sets. For certain simple sets (namely sets with a single `tentacle'), we give complete answers to our questions \ref{one}--\ref{four}, and a partial solution to the moment problem. We provide a systematic way to produce more examples and counterexamples to our questions. The results show that in principle anything can happen, even for regular sets in the plane. The methods are based on the study of key forms for semidegree functions and corresponding (algebraic) compactifications of $\cc^2$, mostly from \cite{contractibility, non-negative-valuation, sub1}. In Section \ref{bound} we show that the algebra $\B(S)$ can always be interpreted as the algebra of bounded polynomials on 
 a higher dimensional set. In this way, we get more three-dimensional and explicit counterexamples to Question \ref{three}.

We also remark that it is possible to extend the analysis in Sections \ref{single-subsection} and \ref{general-subsection} of the `subdegrees' associated to planar sets to higher dimensional subalgebraic sets. It is however more technical in nature and will be part of a future work. 

\section{Standard tentacles}\label{tent}
In this section we prove that Question \ref{one}, the strongest of the above questions, has a positive answer for a large class of sets. We recall the definition of a {\it standard tentacle} from \cite{ne}.
\begin{definition}
For $z=(z_1,\ldots,z_n)\in\Z^n,$ a {\it standard $z$-tentacle} is a set $$\left\{ \lambda^zb:=(\lambda^{z_1}b_1,\ldots,\lambda^{z_n}b_n)\mid \lambda\geq 1, b\in B\right\}$$ where $B\subseteq (\R\setminus\{0\})^n$ is a compact semialgebraic set with nonempty interior. We call $B$ a {\it base} of the tentacle.
\end{definition}

\noindent
Any  $z\in \Z^n$  defines a weighted degree $\deg_z$ on $\R[x],$ by  assigning degree $z_i$ to the variable $x_i.$ We set $$\varphi_z:= \max\{0, z_1,\ldots,z_n\}$$ and note that $\max \{ \deg_z(p)\mid p\in \R[x]_d\}= d\cdot \varphi_z.$

For finite unions of standard tentacles, the modules $\B_d(S)$ from the last section can be described via these weighted degrees. For this let $z^{(1)},\ldots, z^{(m)}\in \Z^n$ be given. We write $\deg_i$ instead of $\deg_{z^{(i)}}$ and $\varphi_i$ instead of $\varphi_{z^{(i)}}$.

\begin{proposition}\label{tendeg} Assume $S\subseteq \R^n$ is a finite union of standard tentacles, corresponding to the directions $z^{(1)},\ldots,z^{(m)}\in\Z^n.$ Then for all $d\in\N$ $$\B_d(S)= \left\{ p\in \R[x] \mid \deg_i(p) \leq d\cdot \varphi_i \quad \mbox{for } i=1,\ldots,m\right\}.$$ In particular, $\B_d(S)$ is spanned as a vector space by the monomials contained in it.
\end{proposition}
\begin{proof}"$\subseteq$": Let $p\in \B_d(S)$ with $p^2\leq q$ on $S$, for some $q\in \R[x]_{2d}.$ Since $q(\lambda^{z^{(i)}}b)$ is of degree at most $2d\cdot \varphi_i$ in $\lambda,$  it follows that $\deg_i(p)\leq d\cdot \varphi_i.$  At this point we need that tentacles have nonempty interior; there is some curve $\lambda^{z^{(i)}}b$ in the tentacle, on which $p$ grows with $\deg_i(p).$
"$\supseteq$": Since $\B_d(S)$ is a vector space, we can assume that $p=x^\beta$ is a monomial. We can also assume $m=1.$  We have for $\lambda\geq 1$ $$p^2(\lambda^{z^{(1)}}b)=b^{2\beta}\lambda^{2z^{(1)}\beta^t}\leq C\lambda^{2\varphi_1 \cdot d},$$  for all $b$ in the base $B$ of the tentacle. Choose $\alpha$ with $\vert\alpha\vert \leq d$ and $z^{(1)}\alpha^t=\varphi_1 \cdot d.$ Then choose $D>0$ such that $C\leq D^2b^{2\alpha}$ for all $b\in B.$ We obtain $p^2\leq (Dx^{\alpha})^2$ on $S$, and thus $p\in \B_d(S)$. \end{proof}

In \cite{ne} Theorem 5.4 it was shown that Question \ref{five} has a positive answer, if $S$ is a finite union of standard tentacles. We improve upon this now, while also simplifying the proof significantly:

\begin{theorem}
If $S$ is a finite union of standard tentacles, then $\B(S)$ is finitely generated.
\end{theorem}
\begin{proof} We use the same notation as before. Consider the set $$M:=\{ (\alpha,d)\in\N^{n+1}\mid z^{(i)}\alpha^t\leq d\cdot \varphi_i  \mbox{ for }  i=1,\ldots,m\}.$$ By Proposition \ref{tendeg}, a polynomial from $\R[x,t]$ belongs to $\B(S)$ if and only of all its monomial do, and a monomial $x^\alpha t^d$ belongs to $\B(S)$ if and only if $(\alpha,d)\in M$. The lattice points in  a rational convex cone form a finitely generated semigroup, by a well-known result of Hilbert. So if $(\alpha_1,d_1),\ldots,(\alpha_r,d_r)$ generate $M$, then the monomials $x^{\alpha_i}t^{d_i}$ generate $\B(S)$.
\end{proof}

\section{A first counterexample}\label{count}

In this section we construct a first {\it regular} semialgebraic set $S$, for which Question \ref{four}, and thus all the other questions as well, have a negative answer. We will see more examples later, but for this one we give a completely elementary and constructive proof. 
We consider two sets $$S_1=\left\{ (a,b)\in\R^2\mid 1\leq a, 1\leq a^3b+ a^6 -a \leq 2\right\}$$
$$S_2=\left\{ (a,b)\in\R^2\mid 1\leq a, 1\leq a^3b-a^6-a\leq 2\right\}$$
and set $S:=S_1\cup S_2.$ Note that $S$ is basic closed semialgebraic, i.e. definable by finitely many simultaneous polynomial inequalities. In fact if $p=x^3y+x^6-x$ and $q=x^3y-x^6-x$, then $S=\mathcal S(x-1, -(2-p)(p-1)(2-q)(q-1)).$
\bigskip
\begin{center}
\includegraphics[scale=0.4]{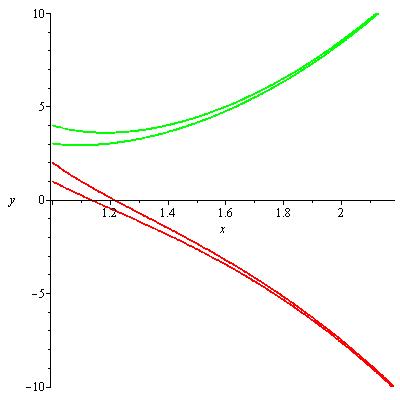}
\end{center}

\begin{theorem}
In the above example we have $\B_0(S)=\R$, but already $\B_1(S)$ is of infinite dimension.
\end{theorem}
\begin{proof}
First note that $S_1$ consists precisely of the points $(\lambda, r\lambda^{-3}-\lambda^3+\lambda^{-2})$ for $\lambda\geq 1$ and $r\in [1,2].$ So if for $p\in \R[x_1,x_2]$ the $\lambda$-degree of $p(\lambda, r\lambda^{-3}-\lambda^3+\lambda^{-2})$ is at most $d$, then $p$ belongs to $\B_d(S_1)$. In fact, $p^2$ can then be bounded by some $C+Dx_1^{2d}$ on $S_1$. Note that $\B_0(S_1)$ consists precisely of those $p$ where this $\lambda$-degree is $\leq 0.$ 

If $q(x,y):=p(x,y-x^3+x^{-2})$, then $p(\lambda, r\lambda^{-3} -\lambda^3+\lambda^{-2})=q(\lambda, r\lambda^{-3}).$ So the $\lambda$-degree of $p(\lambda,r\lambda^{-3}-\lambda^3+\lambda^{-2})$ equals  $\deg_{(1,-3)}$ of the Laurent polynomial $p(x,y-x^3+x^{-2}).$
For $S_2$ the same is true with $+x^3$ instead of $-x^3$ everywhere. The claim will thus follow if we construct polynomials $p\in\R[x,y]$ of arbitrarily high degree, such that the $(1,-3)$-degree of both Laurent polynomials $p_+=p(x,y+x^3+x^{-2})$ and $p_-=p(x,y-x^3+x^{-2})$ is at most $1,$ and if we show that degree $\leq 0$ is only possible for constant polynomials $p$.


First consider $q=y^2-x^6.$ We find $$q_\pm=q(x,y\pm x^3+x^{-2})= \pm 2x \pm 2x^3y+x^{-4}+2x^{-2}y+y^2.$$ Using this, we next consider $r=x^ky^l(y^2-x^6)^m$ and the Laurent polynomials $r_\pm=r(x,y\pm x^3+x^{-2})$. We find \begin{align*} r_\pm = \sum_{a+b+c=l,d+e+f+g+h=m} \frac{l!m!}{a!b!c!d!e!f!g!h!}(\pm1)^{b+d+e}2^{d+e+g}x^{k+3b-2c+d+3e-4f-2g}y^{a+e+g+2h}\end{align*}
From this formula we can read off the following facts:
\begin{itemize}
\item The coefficients of $r_+$ and $r_-$ are the same up to signs. In fact, whether $b+d+e$ is even or odd only depends on the monomial $x^{k+3b-2c+d+3e-4f-2g}y^{a+e+g+2h}$ (in fact only on $k+3b-2c+d+3e-4f-2g$).

\item The Newton polytope of $r_\pm$ has vertices $$(-4m-2l+k,0),(k,2m+l),(m+3l+k,0),(3m+3l+k,m).$$ There are monomials on the line from $(m+3l+k,0)$ to $(3m+3l+k,m)$, and on parallel lines shifted by $5$ to the left. No other monomials occur. This can be seen by checking that the $(1,-2)$-degree of the monomial $x^{k+3b-2c+d+3e-4f-2g}y^{a+e+g+2h}$  is $$k-2l+m +5(b-f-g-h).$$

\item The signs of the coefficients of $r_-$ and $r_+$ obey the following rule. On the line through $(m+3l+k,0)$ and $(3m+3l+k,m)$ the signs differ by $(-1)^{m+l}$. Going through parallel lines in steps of $5$ to the left, the sign change  oscillates from $+$ to $-$. 
\end{itemize}
We are now ready to construct the desired polynomials. We start with $$p^{(1)}=(y^2-x^6)^m$$ for an arbitrarily large $m$. Many of the monomials of $p^{(1)}_+$ and $p^{(1)}_-$ have $(1,-3)$-degree $\leq 1$ anyway. However, there are some which don't. If in the Newton polytope we follow the line from $(3m,m)$ in direction towards $(m,0$), the first monomial $x^{3m}y^m$ is of degree $0$, and the second monomial $x^{3m-2}y^{m-1}$ is of degree $1$. We can tolerate both of them. The next one is however $x^{3m-4}y^{m-2}$, and here we have a $(1,-3)$-degree of $2$. We now modify $p^{(1)}$ by adding $$p^{(2)}=c\cdot x^2(y^2-x^6)^{m-2}$$ to $p^{(1)}$, with a suitable coefficient $c$. Since $p^{(2)}_\pm$ gives rise to a Newton polytope with vertices $$(-4m+10,0), (2,2m-4)),(m,0),(3m-4,m-2),$$ we can choose $c$ to cancel the monomial $x^{3m-4}y^{m-2}$ in {\it both} $p^{(1)}_+$ and $p^{(1)}_-$ {\it at the same time}. This follows from the above  sign considerations: the coefficients in $
 p_+^{(1)}$ and $p_-^{(1)}$ differ by $(-1)^{m}$ at this monomial, and the same is true for $p^{(2)}_+$ and $p^{(2)}_-$. 

Now the new Laurent polynomials $(p^{(1)}+p^{(2)})_\pm$ both have one less of the bad monomials, namely  $x^{3m-4}y^{m-2}$. At the same time, no new monomials arise. The coefficients are still the same in absolute value, and whether the sign changes is determined by the same rule as described above.

 In this way one proceeds: Assume that all monomials up to $(3m-2(i-1),m-(i-1))$ have already been cancelled. Write $i=3l+k$ with $0\leq k\leq 4$ and $i-l$ even. Then a term $$p^{(i)}=c\cdot x^ky^l(y^2-x^6)^{m-i}$$ will allow to also cancel the monomial $(3m-2i,m-i)$ in both Laurent polynomials $$(p^{(1)}+\cdots + p^{i-1)})_\pm$$ at the same time. This follows from the fact that $(-1)^{m-i+l}=(-1)^m$ since $i-l$ is even. 

Once all bad monomials on this line are cancelled, one resumes with bad monomials on parallel lines to the left in a similar fashion. On the next line to the left, one for example has to write $i-5=3l+k$, this time $i-l$ odd, and so on...

When trying to figure out why the degree $1$ monomials, for example  $x^{3m-2}y^{m-1}$, cannot be cancelled in this way, one  sees why only constant polynomials $p$ give rise to $p_\pm$ both of degree $\leq 0.$ One checks that the highest degree part in the $(1,3)$-grading of some $p\in\R[x,y]$ with $\deg_{(1,-3)}p_\pm \leq 0$ must be $(y^2-x^6)^m$ for some $m$, up to scaling. In fact the linear equations for monomials on the maximal $(1,3)$-line in the Newton polytope of $p_\pm$ are linearly independent, having Pascal matrices as coefficient matrices.  Since there are $2m+1$ coefficients and $2m$ equations, there is a unique solution up to scaling.

Now the monomial $x^{3m-2}y^{m-1}$ in $(y^2-x^6)^m_\pm$ has sign $(\pm 1)^m$, as we have seen. It can only cancel with terms coming from  $x(y^2-k^6)^{m-1},$ but here the sign will be $(\pm1)^{m-1}$. So the cancellation cannot work for both substitutions at the same time. This finishes the proof.
\end{proof}

\begin{rem}
Proposition \ref{almost-single} below gives an `explanation' of the above counter example in the language of Section \ref{plane}.
\end{rem}

As explained in the introduction, none of the existing methods to decide the moment problem seems to work for  sets of this kind. The reduction result from \cite{sch2} cannot be applied since $\B_0(S)=\R$, and since Question \ref{five} has a negative answer, the usual way to see that the moment problem is unsolvable is also not successfull. Sets of this kind seem to call for completely new methods.

\section{Sets in the plane}\label{plane}

In this section we use more elaborate techniques, mostly from \cite{contractibility, non-negative-valuation, sub1}, to examine planar sets in more detail. We will obtain many more examples and counterexamples to our question (see Example \ref{exex}).

\subsection{Degree like functions associated to a subset of $\rr^n$}
Let $S$ be a subset of $\rr^n$ and $\delta_S: \rr[x_1, \ldots, x_n]\setminus\{0\} \to \zz$ be the function that maps $f$ to the smallest $d$ such that $f \in \scrB_d(S)$. Then $\delta_S$ is a {\em degree-like function} in the terminology of \cite{sub1} (or equivalently, $-\delta_S$ is an {\em order function} in the terminology of \cite{szpiro}), i.e.\ $\delta_S$ satisfies 
\begin{enumerate}
\let\oldenumi\theenumi
\renewcommand{\theenumi}{P\oldenumi}
\item \label{ative} $\delta_S(f+g) \leq \max\{\delta_S(f),\delta_S(g)\}$ with equality if $\delta_S(f) \neq \delta_S(g)$, and
\item \label{mtive} $\delta_S(fg) \leq \delta_S(f) + \delta_S(g)$.
\end{enumerate}
Some trivial observations are 
$\scrB(S) = \dsum_{d \geq 0} \{f: \delta_S(f) \leq d\}$ and $\delta_S \leq \deg$.
Now define $\bar \delta_S: \rr[x_1, \ldots, x_n]\setminus\{0\} \to \rr$ as 
\begin{align*}
\bar \delta_S(f) := \lim_{n \to \infty} \delta_S(f^n)/n.
\end{align*}
It is not hard to see that $\bar \delta_S$ is well defined and satisfies $\bar \delta_S(f^k) = k\bar\delta_S(f)$ for all $f$ and $k$. We call $\bar\delta_S$ the {\it normalization} of $\delta_S$. In the terminology of \cite{szpiro}, $-\bar \delta_S$ is a {\em homogeneous order function}. We will examine the structure of $\bar \delta_S$ in more details for the case $n=2$.

\begin{defn}
Let $S$ be a semi-algebraic subset of $\rr^2$. For each $r > 0$, let $B_r$ be the ball of radius $r$ centered at the origin. For large enough $r$, the number of connected components of $S \setminus B_r$ becomes stable. Each of these components is called a {\em tentacle} of $S$. 
\end{defn}

\subsection{The case of a single tentacle} \label{single-subsection}
Throughout this subsection we assume that $S$ is a semi-algebraic subset of $\rr^2$ such that
\begin{enumerate}
\let\oldenumi\theenumi
\renewcommand{\theenumi}{A\oldenumi}
\item \label{single-assumption} $S$ has only one tentacle, and
\item \label{regular-assumption} the tentacle of $S$ is regular, i.e.\ it contains a dense open subset.
\end{enumerate}
Let $\bar S$ be the closure of $S$ in $\rr\pp^2$ and $L_\infty$ be the line at infinity on $\rr\pp^2$.

\begin{lemma} \label{single-degree-lemma}
If $\bar S$ intersects $L_\infty$ at more than one point, then $\delta_S = \bar \delta_{S} = \deg$.
\end{lemma}

\begin{proof}
Since $S$ has only one tentacle it follows that $\bar S \cap L_\infty$ is connected. It follows that $|\bar S \cap L_\infty| = \infty$. Choose coordinates $[X:Y:Z]$ on $\rr\pp^2$ such that $(x,y) := (X/Z,Y/Z)$ are coordinates on $\rr^2$. Choose an infinite sequence of points $P_i := [1:c_i:0] \in \bar S \cap L_\infty$ and curves $C_i \subseteq S$ such that $P_i \in \bar C_i$. Then $(1/x,y/x)$ are coordinates near each $P_i$ and $\bar C_i$ has a Puiseux expansion at $P_i$ of the form $y/x = c_i + \sum_{q \in \Q, q>0} c_{i,q}(1/x)^q$, or equivalently, of the form $y = c_i x + \sum_{q \in \Q, q<1} c_{i,q}x^q$. Now pick two polynomials $g_1, g_2 \in \rr[x,y]$ with $d_1 := \deg(g_1) < d_2 := \deg(g_2)$. Then 
\begin{align*}
g_i|_{C_j} = g_i(x,y)|_{y = c_j x + \sum_{q \in \Q, q<1} c_{j,q}x^q} = g_{i,d_i}(1,c_j)x^{d_i} + \lot,
\end{align*}
where $g_{i,d_i}$ is the leading form of $g_i$ and $\lot$ stands for `lower order terms' (in $x$). If $g_{2,d_2}$ is non-negative on $\rr^2$, then it follows that for generic $C_j$, $g_2|_{C_j} > g_1|_{C_j}$ for sufficiently large $|x|$. It therefore follows that $\delta_S = \deg$, as required.
\end{proof}

Now assume \ref{single-assumption} and \ref{regular-assumption} hold and that $\bar S$ intersects $L_\infty$ at only one point $P$. Choose a linear function $u$ such that $P$ is not on the closure of the line $u = 0$. Then (without changing $\delta_S$), we may assume that for sufficiently large values of $|u|$, all points of $S$ are bounded by curves $C_i := \{f_i(x,y) = 0\}$, $1 \leq i \leq 2$. Choose another linear function $v$ such that $(u,v)$ is linearly independent. Then $(1/u,v/u)$ is a set of coordinates (on $\rr\pp^2$) near $P$ and (the closure of) each $C_i$ has a Puiseux expansion at $P$ of the form $v/u = \sum_{j\geq 0} a_{ij} (1/u)^{\tilde \omega_{ij}}$, with $0 \leq \tilde \omega_{i0} < \tilde \omega_{i1} < \cdots$ and $a_{ij} \in \rr$, or $v = \phi_i(u)$, where $\phi_i(u) := \sum_{j\geq 0} a_{ij} u^{\omega_{ij}}$ where $\omega_{ij} := 1 - \tilde \omega_{ij}$.

\begin{lemma} \label{single-tentacle-lemma}
Let $\omega$ be the largest (rational) number such that the coefficients of $u^{\omega'}$ in the expansion of $C_i$'s are equal for all $\omega' > \omega$. Let $\phi(u)$ be the (common) part of $\phi_i$'s consisting of all terms with the exponent of $u$ greater than $\omega$. Let $\xi$ be a new indeterminate and define $\delta^*_S : \rr[x,y]\setminus\{0\} \to \Q$ as
\begin{align}
\delta^*_S(f) := \deg_u(f(u,\phi(u) + \xi u^\omega)). \label{delta^*_S}
\end{align} 
Then
\begin{enumerate}
\item \label{single-1} $\bar \delta_S = \max\{0,\delta^*_S\}$.
\item \label{single-2} $\delta_S = \lceil \bar\delta_S \rceil$.
\item \label{single-3} For each $f \in \rr[x,y]\setminus \{0\}$, $f/|u|^{\delta^*_S(f)}$ is bounded outside a compact set on $S$. 
\end{enumerate}
\end{lemma}

\begin{proof}
At first we claim that $\delta_S(u) = \bar \delta_S(u) = 1$. Indeed, since $\deg(u) = 1$, if the claim does not hold, then $\bar\delta_S(u) < 1$ and therefore there is a positive integer $d$ and a polynomial $h \in \rr[u,v] = \rr[x,y]$ with $e := \deg(h) < 2d$ such that $u^{2d} < h$ on $S$. But it is impossible, since 
\begin{align*}
h|_{C_i} = h(u,v)|_{v = \phi_i(u)} = cu^e + \lot
\end{align*}
for some $c\in r$, and therefore $h|_{C_i} < u^{2d}|_{C_i}$ for large enough $|u|$. This proves the claim.\\

For each $t \in [0,1]$, let $\phi_t(u) := t\phi_1(u) + (1-t)\phi_2(u)$. Then for sufficiently large $|u|$, for each $t \in [0,1]$, $ v = \phi_t(u)$ defines a branch of real analytic curve $C_t$ in $S$ such that $\lim_{|u| \to \infty} C_t = P$. Now note that 
\begin{align*}
\phi_t(u) = \phi(u) + (ta_1 + (1-t)a_2)u^\omega + \lot  = \phi(u) + u^\omega\psi_t(u),
\end{align*}
where $a_i$ is the coefficient of $u^\omega$ in $\phi_i$, and $\psi_t(u)$ is (a Puiseux series in $1/u$) of the form $\sum_{\omega' \leq 0} b_{\omega'}(t)u^{\omega'}$ with $b_0(t) = ta_1 + (1-t)a_2$. Let $f \in \rr[u,v]$ and $d := \deg_u(f(u,\phi(u) + \xi u^\omega))$. Then 
\begin{align*}
f(u,\phi(u) + \xi u^\omega) = f_0(\xi)u^d + \lot
\end{align*}
where $f_0$ is a non-zero polynomial in $\xi$. It follows that 
\begin{align*}
f|_{C_t} = f(u,\phi(u) + \xi u^\omega)|_{\xi = \psi_t(u)} = f_0(ta_1 + (1-t)a_2)u^d  + \lot
\end{align*}
We see that $f/|u|^d$ is bounded outside a compact set on $S$ and consequently $\bar \delta_S(f) \leq \max\{0,d\}$ and $\delta_S(f) \leq \max\{0,\lceil d \rceil\}$.\\

On the other hand, if $d > 0$ and $h$ is a polynomial in $\rr[u,v]$ with $\deg(h) < dk$ for some integer $k \geq 1$, then clearly $e := \deg_u(h(u, \phi(u) + \xi u^\omega)) \leq \deg(h) < dk$. Since 
\begin{align*}
h|_{C_t} = h(u,\phi(u) + \xi u^\omega)|_{\xi = \psi_t(u)} = h_0(ta_1 + (1-t)a_2)u^e  + \lot
\end{align*}
for some polynomial $h_0 \in \rr[\xi]$, it follows that $f^{2k}$ eventually becomes bigger on each $C_t$ than $h_0^2$. It follows that $\bar \delta_S(f) \geq d$, and consequently, $\bar \delta_S(f) = d$, as required to prove the first assertion of the lemma. This same argument with $k = 1$ in fact also proves the second assertion. The last assertion follows from the last sentence of the preceding paragraph. 
\end{proof}

\begin{defn}\label{deg-wise-puiseux}
$\phi(u) + \xi u^\omega$ from identity \eqref{delta^*_S} is called the {\em generic degree-wise Puiseux series} corresponding to $S$.
\end{defn}

\begin{rem} \label{semi-remark}
Note that $\delta^*_S$ is a {\em semidegree} (in the terminology of \cite{sub1}), i.e.\ $\delta^*_S$ satisfies Property \ref{mtive} of degree-like functions with an equality.
\end{rem}

\begin{cor}
Assume \ref{single-assumption} and \ref{regular-assumption} hold. Then $\bar \delta_S$ is a {\em semidegree} iff $\bar \delta_S = \bar \delta^*_S$ iff $\bar \delta^*_S$ is non-negative on $\rr[x,y]$.
\end{cor}

\subsection{The general regular case} \label{general-subsection}
 
\begin{prop} \label{general-2-prop}
Let $S$ be a semialgebraic subset of $\rr^2$ such that every tentacle of $S$ satisfies property \ref{regular-assumption}. Then 
\begin{align}
\bar \delta_S   = \max\{\bar \delta_T: T\ \text{is a tentacle of}\ S\} 
                = \max\left( \{0\} \cup \{\delta^*_T: T\ \text{is a tentacle of}\ S\} \right).
\end{align} \label{bar-delta-presentation}
\end{prop}
\begin{proof}
It is clear that $\bar \delta_T \leq \bar \delta_S$ for every tentacle $T$ of $S$, which proves that LHS $\geq$ RHS in \eqref{bar-delta-presentation}. The $\leq$ direction follows from the last assertion of Lemma \ref{single-tentacle-lemma}.
\end{proof}

\begin{cor}
Let $S$ be as in Proposition \ref{general-2-prop}. Then $\bar \delta_S$ is a {\em subdegree} (in the terminology of \cite{sub1}), i.e.\ $\bar \delta_S$ is the maximum of finitely many semidegrees.
\end{cor}

\begin{cor}
Let $S$ be as in Proposition \ref{general-2-prop}. Then 
\begin{enumerate}
\item \label{integral-1} There exists a positive integer $N$ such that $\bar \delta_S(f) \in \frac{1}{N}\zz$ for all $f \in \rr[x,y]\setminus\{0\}$.
\item \label{integral-2} Define
\begin{align*}
\bar \scrB_d(S) &:= \{f \in \rr[x,y]: \bar \delta_S(f) \leq d\}, \\
\bar \scrB(S) &:= \dsum_{d \geq 0} \bar \scrB_d(S). 
\end{align*}
Then $\scrB_d(S) = \bar \scrB_d(S)$ for all $d \geq 0$ and $\scrB(S) =\bar \scrB(S)$. 
\end{enumerate}
\end{cor}

\begin{proof}
Assertion \ref{integral-1} follows immediately from the Lemma \ref{single-tentacle-lemma}, Proposition \ref{general-2-prop} and the observation that a semialgebraic set has only finitely many tentacles. Assertion \ref{integral-2} follows from Proposition \ref{general-2-prop} and Assertion \ref{single-2} of Lemma \ref{single-tentacle-lemma}. 
\end{proof}

\subsection{Semidegrees on $\cc[x,y]$ and corresponding compactifications of $\cc^2$}
\subsubsection{Background} \label{background-subsection}

Let $\delta$ be a semidegree (see Remark \ref{semi-remark}) on $\cc[x,y]$ defined as:
\begin{align}
\delta(f) := \deg_x(f(x,\phi(x) + \xi x^\omega)). \label{delta-defn}
\end{align} 
where $\xi$ is an indeterminate, $\phi \in \cc[x^{1/N}, x^{-1/N}]$ for some positive integer $N$ and $\omega \in \Q$, $\omega < \ord_x(\phi)$. Associated to $\delta$ there is a finite sequence of elements in $\cc[x,x^{-1},y]$ called the {\em key forms} of $\delta$ (see \cite[Definition 3.17]{contractibility}). The sequence starts with $f_0 := x, f_1 := y$, and continues until there is an element in $\cc[x,x^{-1},y]$ which can be expressed as a polynomial in the computed key forms and whose $\delta$-value is smaller than the `expected' value. An algorithm and detailed example for the computation of key forms of $\delta$ from $\phi$ and $\omega$ appears in \cite[Section 3.3]{contractibility}.

\begin{example} 
\mbox{}\\
\begin{tabular}{cc}
$\phi(x) + \xi x^\omega$ & key forms\\
\hline 
$\xi x^{p/q}$ & $x,y$ \\
$cx^\frac{p}{q} + \xi x^\omega$, $p,q$ rel.\ prime integers, $q > 0$, $\omega < \frac{p}{q}$ & $x,y, y^q - c^qx^p$ \\
$x^{5/2} + x^{-3/2} + \xi x^{-5/2}$ & $x,y,y^2 - x^5, y^2 - x^5 - 2x$\\
$x^{5/2} + x^{-1} + x^{-3/2} + \xi x^{-5/2}$ & $x,y,y^2 - x^5, y^2 - x^5 - 2x^{-1}y, y^2 - x^5 - 2x^{-1}y - 2x$
\end{tabular}
\end{example}

Given a (normal) algebraic variety $Y$ and a codimension one irreducible subvariety $V$ of $Y$, the {\em order of pole} along $V$ defines a semidegree on the field of rational functions on $Y$. Given a semidegree $\delta$ on $\cc[x,y]$, the following proposition gives the construction of a compact algebraic variety containing $\cc^2$ which `realizes' $\delta$ (as the order of pole) along some curve. 

\begin{prop}[{see \cite[Proposition 2.10]{non-negative-valuation}}] \label{compact-prop}
Let $\delta$ be defined as in \eqref{delta-defn}. Assume that $\delta \neq \deg$. Pick the smallest positive integer $N$ such that $N\delta$ is integer-valued. Then there exists a unique compactification $\bar X$ of $X := \cc^2$ such that 
\begin{enumerate}
\item $\bar X$ is projective and normal.
\item $\bar X_\infty := \bar X \setminus X$ has two irreducible components $C_1,C_2$.
\item The semidegree on $\cc[x,y]$ corresponding to $C_1$ and $C_2$ are respectively $\deg$ and $N\delta$. 
\end{enumerate}
Moreover, all singularities of $\bar X$ are {\em rational}.
\end{prop}

\begin{theorem}[Characterizing when $\delta$ is non-negative or positive {\cite[Theorem 1.4]{non-negative-valuation}}] \label{positive-thm}
Let $\delta$ be a semidegree on $\cc[x,y]$ and let $g_0, \ldots, g_{n+1}$ be the key forms of $\delta$ in $(x,y)$-coordinates. Then 
\begin{enumerate}
\item \label{non-positive-assertion} $\delta$ is non-negative on $\cc[x,y]\setminus \cc$ iff $\delta(g_{n+1})$ is non-negative.
\item \label{positive-assertion} $\delta$ is positive on $\cc[x,y]\setminus \cc$ iff one of the following holds:
\begin{enumerate}
\item $\delta(g_{n+1})$ is positive,
\item \label{almost-zero} $\delta(g_{n+1}) = 0$ and $g_k \not\in \cc[x,y]$ for some $k$, $0 \leq k \leq n+1$, or 
\addtocounter{enumii}{-1}
\let\oldenumii\theenumii
\renewcommand{\theenumii}{$\oldenumii'$}
\item \label{almost-zero'} $\delta(g_{n+1}) = 0$ and $g_{n+1} \not\in \cc[x,y]$.
\end{enumerate}
\end{enumerate}
Moreover, conditions \ref{almost-zero} and \ref{almost-zero'} are equivalent.
\end{theorem}

To a semidegree $\delta$ on $\cc[x,y]$ we associate a graded ring 
\begin{align}
\cc[x,y]^\delta := \dsum_{d \geq 0} \{f: \delta(f) \leq d\}.
\end{align}
In the case $\delta$ is realized (as in Proposition \ref{compact-prop}) as the order of pole along a curve on a normal surface, $\cc[x,y]^\delta$ can be interpreted as the graded ring of global sections of a divisor. The following results exploit this connection to study finiteness properties of $\cc[x,y]^\delta$.

\begin{prop} \label{non-finite-delta}
Let $f_\delta$ be the last key form of $\delta$. Assume $\delta$ is positive on $\cc[x,y]\setminus \cc$ and $f_\delta$ is {\em not} a polynomial. Then $\cc[x,y]^\delta$ is {\em not} finitely generated over $\cc$. Moreover, 
\begin{enumerate}
\item If $\delta(f_\delta) > 0$, then $\cc[x,y]^\delta_d :=  \{f: \delta(f) \leq d\}$ is a finite dimensional vector space over $\cc$ for all $d$.
\item If $\delta(f_\delta) = 0$, then there exists $d > 0$ such that $\cc[x,y]^\delta_d$ is an infinite dimensional vector space over $\cc$.
\end{enumerate}
\end{prop}

\begin{proof}
It follows from the assumptions that $\delta \neq \deg$. Let $\bar X$ be the compactification of $X := \cc^2$ from Proposition \ref{compact-prop}. Since $\delta$ is positive on $\cc[x,y]\setminus \cc$, Theorem \ref{positive-thm} implies that $\delta(f_\delta) \geq 0$. At first assume $\delta(f_\delta) > 0$. It then follows by the same arguments as in the proof of \cite[Theorem 1.14]{contractibility} that $\cc[x,y]^\delta$ is not finitely generated over $\cc$ and that there exists a divisor $D$ on $\bar X$ such that 
\begin{align*}
\cc[x,y]^\delta = \dsum_{d \geq 0} H^0(\bar X, \mathcal{O}_{\bar X}(dD)).
\end{align*} 
Since $H^0(\bar X, \mathcal{O}_{\bar X}(dD))$ is a finite dimensional vector space over $\cc$ for each $d$, this proves the proposition for the case that $\delta(f_\delta) > 0$. \\

Now assume $\delta(f_\delta) = 0$. It then follows from \cite[identity (11)]{non-negative-valuation} that $(C_1,C_1) = 0$. Since the singularities of $\bar X$ are rational, this implies that $C_1$ and $C_2$ are $\Q$-Cartier divisors. It follows that $D_k := kC_1 + C_2$ is a nef ($\Q$-Cartier) divisor on $\bar X$ for all $k \gg 0$. Pick any positive integer $e$ such that $eC_i$ is a Cartier divisor for each $i$. Then $eD_k$ is an ample Cartier divisor on $\bar X$ for all $k \gg 0$. Let $\pi:\tilde X \to \bar X$ be a resolution of singularities of $\bar X$. Let $H$ be a fixed ample divisor and $K_{\tilde X}$ be the canonical divisor on $\tilde X$. W.l.o.g.\ we may assume that the supports of both $H$ and $K_X$ are contained in $\tilde X \setminus X$. Since $H + \pi^*(eD_k)$ is also ample for each $k$, it follows from a theorem of Reider (see e.g.\ \cite{lazarsfeld}) that $\tilde D_{k} := 3H + \pi^*(3eD_k) + K_{\tilde X}$ is base-point free for each $k$. Let $c_1$ (resp.\ $c_2$)
  be the coefficient of $C_1$ (resp.\ $C_2$) in $3H + K_{\tilde X}$. Since $\tilde D_{k}$ is base-point free, there exists $f_k \in \cc[x,y]$ such that $\deg(f_k) = c_1 + 3ek$ and $\delta(f_k) \leq c_2 + 3e$. Let $d := c_2 + 3e$. It follows that the degree $d$ part of $\cc[x,y]^\delta$ is infinite dimensional over $\cc$, as required.
\end{proof}

\begin{prop}
Let $f_\delta$ be the last key form of $\delta$. Assume one of the following conditions hold:
\begin{enumerate}
\item {\em all} key forms of $\delta$ are polynomials (equivalently, $f_\delta$ is a polynomial), or
\item $\delta(f_\delta) < 0$ (equivalently, there is $f \in \cc[x,y]\setminus\{0\}$ such that $\delta(f) < 0$).
\end{enumerate} 
Then $\cc[x,y]^\delta$ is finitely generated over $\cc$.
\end{prop}

\begin{proof}
Let the key forms of $\delta$ be $f_0 = x, f_1 = y, f_2, \ldots, f_l$. At first we assume that all $f_k$'s are polynomials. Pick a positive integer $N$ such that $N\delta$ is integer-valued. It suffices to show that $\cc[x,y]^{N\delta}$ is finitely generated over $\cc$, where 
\begin{align*}
\cc[x,y]^{N\delta} = \dsum_{d \geq 0} \{f: N\delta(f) \leq d\},
\end{align*}
(since $\cc[x,y]^{N\delta}$ is integral over $\cc[x,y]^{\delta}$). Let $e_j := N\delta(f_j)$, $1 \leq j \leq l$. Let us denote by $(f_j)_{e_j}$ the `copy' of $f_j$ in the $e_j$-th graded component of $\cc[x,y]^{N\delta}$. We now show that $\cc[x,y]^{N\delta}$ is generated as a $\cc$-algebra by $\{(1)_1\} \cup \{(f_j)_{e_j}: 0 \leq j \leq l\}$. Indeed, pick $f \in \cc[x,y]$. Recall (from \cite[Proposition 3.28]{contractibility}) that for each $j \geq 1$, $f_j$ is monic in $y$ (as a polynomial in $y$ with coefficients in $\cc[x]$) and $\deg_y(f_j)$ divides $\deg_y(f_{j+1})$ for $1 \leq j \leq l-1$. It follows that given an $f \in \cc[x,y]$, $f$ has an expression of the form 
\begin{align*}
f = \sum_{\alpha \in \zz_{\geq 0}^{l+2}} a_\alpha x^{\alpha_0}f_1^{\alpha_1} \cdots f_l^{\alpha_l},
\end{align*}
for $a_\alpha \in \cc$ and $\alpha_j < \deg_y(f_{j+1})/\deg_y(f_j)$ for $1 \leq j \leq l-1$. \cite[Theorem 16.1]{maclane-key} then implies that $N\delta(f) = \max\{N\delta(x^{\alpha_0}f_1^{\alpha_1} \cdots f_l^{\alpha_l}): a_\alpha \neq 0\}$. It then immediately follows that
\begin{align*}
(f)_{N\delta(f)} = \sum_{\alpha \in \zz_{\geq 0}^{l+2}} a_\alpha \left((x)_{e_0}\right)^{\alpha_0}\left((f_1)_{e_1}\right)^{\alpha_1} \cdots \left((f_l)_{e_l}\right)^{\alpha_l} \left((1)_1\right)^{N\delta(f) - \sum \alpha_je_j},
\end{align*}
as required to show that $\cc[x,y]^{N\delta}$ is finitely generated as an algebra over $\cc$.\\

Now assume that $\delta(f_l) < 0$. W.l.o.g.\ we may (and will) also assume that $f_l$ is {\em not} a polynomial. Define a map $\nu: \cc[x,y]\setminus\{0\} \to \zz^2$ as follows: for every $g \in \cc[x,y]\setminus\{0\}$, $g|_{y = \phi(x) + \xi x^\omega}$ is of the form $x^\alpha g_\alpha(\xi) + \lot$ for some $\alpha \in \Q$ and $g_\alpha \in \cc[\xi]$ (where $\lot$ denotes `lower-order terms' in $x$). Then set $\nu(g) := (N\alpha, \deg_\xi(g_\alpha))$.

\begin{proclaim} \label{nu-f-claim}
There exists $f \in \cc[x,y]$ such that $\nu(f) = (0,k)$ for some positive integer $k$. 
\end{proclaim}

\begin{proof}
Indeed, \cite[Theorem 1.7]{non-negative-valuation} implies that there exists $h \in \cc[x,y]$ such that $\delta(h) < 0$. Then $f := h^a x^b$ for suitable non-negative integers $a,b$ satisfies the claim.
\end{proof}

Let $G_+ := \{\nu(g): g \in \cc[x,y],\ \delta(g) \geq 0\}$. Then $G_+$ is a sub-semigroup of $\zz_{\geq 0}^2$. Moreover, since $\nu(x)$ is of the form $(k_1,0)$ and $\nu(f)$ (where $f$ is as in Claim \ref{nu-f-claim}) is of the form $(0,k_2)$ for positive integers $k_1,k_2$, it follows that $\zz_{\geq 0}^2$ is {\em integral} over $G_+$, and therefore $G_+$ is a finitely generated semigroup. Pick $g_1, \ldots, g_k$ such that $\nu(g_j)$'s generate $G_+$. Let $d_j := \delta(g_j)$, $1 \leq j \leq k$. The proposition follows from the following claims.

\begin{proclaim} 
$(g_j)_{d_j}$, $1 \leq j \leq k$, generated $\cc[x,y]^\delta$ as an algebra over $\cc[x,y]^\delta_0 :=  \{f: \delta(f) \leq 0\}$.
\end{proclaim}

\begin{proof}
Let $\prec$ be the lexicographic order on $\zz_{\geq 0}^2$. The claim follows from the observation that if $\nu(g) = \sum \alpha_j\nu(g_j)$, then there exists (a unique) $c \in \cc\setminus\{0\}$ such that $\nu(g - cg_1^{\alpha_1} \cdots g_k^{\alpha_k}) \prec \nu(g)$. 
\end{proof}

\begin{proclaim} 
$\cc[x,y]^\delta_0$ is a finitely generated $\cc$-algebra.
\end{proclaim}

\begin{proof}
Note that $\delta \neq \deg$. Let $\bar X$ be the compactification of $\cc^2$ from Proposition \ref{compact-prop}. Since the singularities of $\bar X$ are rational, it follows that $C_1$ is a $\Q$-Cartier divisor. Since $\delta(f_l) < 0$, \cite[Proposition 2.10]{non-negative-valuation} implies that $(C_1,C_1) > 0$. Moreover, since $f_l$ is not a polynomial, \cite[Proposition 2.11]{non-negative-valuation} implies that every compact curve on $\bar X$ intersects $C_1$. It follows from the Nakai-Moishezon criterion that $C_1$ is an ample divisor and therefore $\bar X \setminus C_1$ is an affine variety. The claim follows from the observation that $\cc[x,y]^\delta_0$ is precisely the ring of regular functions on $\bar X \setminus C_1$.
\end{proof}
\renewcommand{\qed}{}
\end{proof}

\begin{cor}[Characterization of finiteness properties  of {$\cc[x,y]^\delta$}] \label{finite-cor}
\mbox{}
\begin{enumerate}
\item \label{finite-generation} $\cc[x,y]^\delta$ is not a finitely generated algebra over $\cc$ iff both of the following conditions holds:
\begin{enumerate}
\item $\delta$ is non-negative on $\cc[x,y]$, and 
\item there is a key form of $\delta$ which is not a polynomial (or equivalently, the last key form of $\delta$ is not a polynomial).
\end{enumerate} 
\item \label{finite-d} $\cc[x,y]^\delta_d$ is a finite dimensional vector space over $\rr$ for all $d \geq 0$ iff $\delta(f_\delta) > 0$, where $f_\delta$ is the last key form of $\delta$.
\end{enumerate} 
\end{cor}

\begin{proof}
The only assertion which does not follow immediately from the preceding results is that $\cc[x,y]^\delta_d$ is a finite dimensional vector space over $\rr$ for all $d \geq 0$ in the case that $\delta(f_\delta) > 0$ and $f_\delta$ is a polynomial. But this follows by the same argument as in the proof of the first assertion of Proposition \ref{non-finite-delta}. 
\end{proof}

\subsubsection{Applications to the single tentacle case}
Throughout this subsection $S$ is assumed to be a semi-algebraic subset of $\rr^2$ which satisfies \ref{single-assumption} and \ref{regular-assumption}. Propositions \ref{semi-effective}--\ref{single-finite}, which are immediate applications of the results from Sections \ref{single-subsection} and \ref{background-subsection}, completely answers Questions \ref{one}--\ref{four} for $S$. Proposition \ref{single-moment} partially solves the moment problem for such $S$.

\begin{prop}[A necessary and sufficient criterion for $\bar \delta_S$ to be a semidegree] \label{semi-effective}
Let $f_S$ be the last key form of $\delta^*_S$. Then the following are equivalent: 
\begin{enumerate}
\item $\bar \delta_S$ is a semidegree.
\item $\delta^*_S$ is non-negative on $\rr[x,y]$.
\item $\delta^*_S(f_S) \geq 0$. 
\end{enumerate}
\end{prop}

\begin{prop}[A necessary and sufficient criterion for $\scrB_0(S) = \rr$] \label{single-0}
Let $f_S$ be the last key form of $\delta^*_S$. Then $\scrB_0(S) = \rr$ iff one of the following holds:
\begin{enumerate}
\item $\delta^*_S(f_S) > 0$, or
\item $\delta^*_S(f_S) = 0$ and $f_S$ is {\em not} a polynomial. 
\end{enumerate}
\end{prop}

\begin{prop}[A necessary and sufficient criterion for the failure of $\scrB(S)$ to be finitely generated] \label{single-finite}
$\scrB(S)$ is not a finitely generated algebra over $\scrB_0(S)$ iff both of the following conditions holds:
\begin{enumerate}
\item $\delta^*_S$ is non-negative on $\rr[x,y]$ (or equivalently, the $\delta^*_S$-value of the last key form of $\delta^*_S$ is non-negative), and 
\item there is a key form of $\delta^*_S$ which is not a polynomial (or equivalently, the last key form of $\delta^*_S$ is not a polynomial).
\end{enumerate} 
Moreover, if $\scrB(S)$ is not a finitely generated algebra over $\scrB_0(S)$, then 
\begin{enumerate}
\item $\scrB_0(S) = \rr$.
\item Let $f_S$ be the last key form of $\delta^*_S$. Then
\begin{enumerate}
\item If $\delta^*_S(f_S) > 0$, then $\scrB_d(S)$ is a finite dimensional vector space over $\rr$ for all $d$.
\item If $\delta^*_S(f_S) = 0$, then there exists $d > 0$ such that $\scrB_d(S)$ is an infinite dimensional vector space over $\rr$.
\end{enumerate}
\end{enumerate}  
\end{prop}

\begin{prop}[Partial solution of the moment problem for $S$] \label{single-moment}
\mbox{}
\begin{enumerate}
\item The moment problem for $S$ cannot be solved by finitely many polynomials in the following cases:
\begin{enumerate}
\item \label{positively-non-solvable} $\delta^*_S(f_S) > 0$, or
\item \label{zero-pol-non-solvable} $\delta^*_S(f_S) = 0$, $f_S$ is a polynomial, and the curve $f_S = \xi$ for generic $\xi$ for generic $\xi$ has genus $\geq 1$. 
\end{enumerate}
\item \label{partial-solvable} Assume one of the following conditions is satisfied:
\begin{enumerate}
\item \label{negatively-solvable} $\delta^*_S(f_S) < 0$, or
\item \label{zero-pol-solvable} $\delta^*_S(f_S) = 0$ and $f_S$ is a polynomial, and the curve $f_S = \xi$ for generic $\xi$ is rational.  
\end{enumerate}
Then there is a compact subset $V$ of $S$ such that the moment problem for $S\setminus V$ can be solved by finitely many polynomials.
\end{enumerate}

\end{prop}

\begin{rem}
The only case not covered by Proposition \ref{single-moment} (for those $S$ which satisfy \ref{single-assumption} and \ref{regular-assumption}) occurs when $\delta^*_S(f_S) = 0$ and $f_S$ is {\em not} a polynomial. As explained at the end of Section \ref{count}, `sets of this kind seem to call for completely new methods.'
\end{rem}

\begin{proof}[Proof of Proposition \ref{single-moment}]
Assertion \ref{positively-non-solvable} follows from Assertion \ref{finite-d} of Corollary \ref{finite-cor} and the discussion following Lemma \ref{four=five}. Assertion \ref{zero-pol-non-solvable} follows from \cite[Corollary 3.10]{posc} coupled with the following observations:  
\begin{enumerate}
\item for all but finitely many $\xi$, the curve $C_\xi := \{f_S = \xi\}$ is smooth (by Bertini's theorem), 
\item $C_\xi$ has only one point at infinity (by \cite[Proposition 4.2]{contractibility}), and
\item there exists a non-empty open interval $I \subseteq \rr$ such that for all $\xi \in \rr$, the point at infinity of $C_\xi$ belongs to the closure (in $\rr\pp^2$) of $S\cap C_\xi$.
\end{enumerate}

For Assertion \ref{partial-solvable}, choose linear coordinates $(u,v)$ on $\rr^2$ such that $u \to \infty$ along the tentacle of $S$ and $(u,v)$ satisfy the hypotheses of Lemma \ref{single-tentacle-lemma} (i.e.\ the point of intersection of the closure $\bar S$ of $S$ in $\rr\pp^2$ and the line at infinity is not on the closure of the line $u = 0$). At first assume we are in the situation of \ref{negatively-solvable}. Then Theorem \ref{positive-thm} implies that there is $f \in \rr[u,v]$ such that $\delta^*_S(f) < 0$. Choose positive integers $a,b$ such that $h := u^a f^b$ satisfies $\delta^*_S(h) = 0$. For each $\xi \in \rr$, let $C_\xi$ be the curve $h= \xi$ and for all $r \geq 0$, let $S_r := \{(u,v) \in S: u \geq r\}$. Then for sufficiently large $r$, we have that
\begin{enumerate}
\item $S_r$ is defined by $\{u \geq r, h_1 \geq 0, h_2 \geq 0\}$, where $h_1, h_2 \in \rr[u,v]$ which `define the boundaries of the tentacle of $S_r$', 
\item $h$ is bounded on $S_r$,
\item $C_\xi$ has a real point at infinity (namely the point in the closure of the line $u=0$) which is not in the closure (in $\rr\pp^2$) of $S_r$,
\item $C_\xi \cap S_r$ does not intersect $\{h_i = 0\}$ for $i \in \{1,2\}$.
\end{enumerate}
Then \cite[Theorem 1]{sch2} and \cite[Theorem 3.11]{posc} imply that $\{u-r,h_1,h_2\}$ solves the moment problem on $S_r$.\\

Now assume the hypotheses of \ref{zero-pol-solvable} are satisfied. Since $C_\xi$ has only one point at infinity for all $\xi$ (by \cite[Proposition 4.2]{contractibility}), it follows from the Abhyankar-Moh-Suzuki theorem (\cite{amoh}, \cite{suzuki}) that $f_S$ is a {\em polynomial coordinate} on $\cc^2$, i.e.\ there exists a polynomial $g \in \cc[u,v]$ such that $\cc[f_S,g] = \cc[u,v]$. It is then not hard to show using Jung's theorem (see e.g.\ \cite{friedland-milnor}) on polynomial automorphisms of the plane that $f_S$ is in fact a polynomial coordinate on $\rr^2$, i.e.\ there exists $g \in \rr[u,v]$ such that $\rr[f_S,g] = \rr[u,v]$. W.l.o.g.\ we may assume that $g \to \infty$ along the tentacle of $S$. Let $S_r := \{(u,v) \in S: g(u,v) \geq r\}$. There exists $h_1, h_2 \in \rr[u,v]$ such that $S_r = \{g \geq r, h_1 \geq 0, h_2 \geq 0\}$ for all sufficiently large $r$. \cite[Theorem 1]{sch2} and \cite[Theorem 2.2]{kuma} then imply that $\{g-r,h_1,h_2\}$ solves the moment 
 problem on $S_r$ for $r$ large enough. This completes the proof of Assertion \ref{partial-solvable}.
\end{proof}

\subsubsection{Explicit construction of single tentacles} \label{construction}
In this subsection we give an algorithm to construct tentacles $S$ with desired behaviour of $\scrB(S)$. The algorithm consists of the construction of a sequence of elements $f_0, \ldots, f_l$, $l \geq 1$, in $\rr[x,x^{-1},y]$ which would be the key forms of the corresponding semidegree $\delta^*_S$. It is a straightforward adaptation of Maclane's construction of key polynomials in \cite{maclane-key}.\\

\paragraph{Initial step:} Set $f_0 := x$ and $f_1 := y$. Pick $\omega_1 \in \Q$ and set $\omega_0 := 1$. \\

\paragraph{Inductive step:} Assume $f_j$'s and $\omega_j$'s have been constructed up to some $k \geq 1$. Let $p_k$ be the {\em smallest} positive integer such that $p_k\omega_k$ is in the additive group generated by $\omega_0, \ldots, \omega_{k-1}$. Then $p_k\omega_k$ can be {\em uniquely} expressed in the form
\begin{align*}
p_k\omega_k = \alpha_{k,0}\omega_0 + \alpha_{k,1} \omega_1 + \cdots + \alpha_{k,k-1}\omega_{k-1}
\end{align*} 
where $\alpha_{k,j}$'s are integers such that $0 \leq \alpha_{k,j} < p_j$ for all $j \geq 1$ (note that there is no restriction on the range of $\alpha_{k,0}$). Pick a non-zero $c_k \in \rr$ and set 
\begin{align*}
f_{k+1} := f_k^{p_k} - c_k \prod_{j=0}^{k-1}f_j^{\alpha_{k,j}}.
\end{align*}
Set $\omega_{k+1}$ to be a rational number less than $p_k\omega_k$. \\

\paragraph{Construction of $S$ from a finite sequence of $f_k$'s:} Assume $f_k$'s and $\omega_k$'s have been constructed up to some $l \geq 1$. Construct $p_l$ and $\alpha_{l,0}, \ldots, \alpha_{l,l-1}$ as in the inductive step. Define   
\begin{align*}
f_{l+1,i} := f_l^{p_l} - c_{l,i} \prod_{k=0}^{l-1}f_k^{\alpha_{l,k}},
\end{align*}
where $c_{l,1}, c_{l,2}$ are distinct real numbers such that each $f_{l+1,i}$ defines a curve $C_i$ on $\rr^2 \setminus y$-axis. 
Let $C^\cc_i$ be the curve defined by $f_{l+1,i}$ in $\cc^2 \setminus y$-axis. Then each $C^\cc_i$ has a unique irreducible branch for which $|x| \to \infty$. It follows that either $C_i$ has a unique branch for which $x \to \infty$, or it has two such branches which come from the same irreducible branch of $C^\cc_i$. In any event, there is a unique `top' branch of $C_i$ for which $x \to \infty$; let us denote it by $C^{top}_i$. Pick $r > 0$ and let $S$ be the region to the right of $x = r$ and bounded by $C^{top}_1$ and $C^{top}_2$. \\

The following is an immediate corollary of the results of the preceding subsection and the observations that $f_0, \ldots, f_l$ are precisely the key-forms of $\delta^*_S$ and $\omega_l = \delta^*_S(f_l)$. Note that the result does not change if we took $S$ to be the region bounded by the `bottom' branches of $C_i$, or if we took corresponding branches of $C_i$ for which $x \to -\infty$. 

\begin{cor} \label{example-cor}
\mbox{}
\begin{enumerate}
\item $\scrB_0(S) = \rr$ iff one of the following is true:
\begin{enumerate}
\item $\omega_l > 0$, or
\item $\omega_l = 0$ and $f_l \not\in \rr[x,y]$ (equivalently, $\omega_l = 0$ and $\alpha_{k,0} < 0$ for some $k$, $1 \leq k \leq l-1$).
\end{enumerate} 
\item $\scrB(S)$ is a finitely generated algebra over $\rr$ iff one of the following conditions hold:
\begin{enumerate}
\item $\omega_l < 0$, or
\item $\omega_l \geq 0$ and $f_k \in \rr[x,y]$ for $1 \leq k \leq l$ (equivalently, $\omega_l \geq 0$ and $\alpha_{k,0} \geq 0$ for $1 \leq k \leq l-1$). 
\end{enumerate} 
\item If $\scrB(S)$ is not a finitely generated algebra over $\rr$, then $\scrB_0(S) = \rr$. Moreover, in this case
\begin{enumerate}
\item if $\omega_l > 0$, then $\scrB_d(S)$ is finite dimensional over $\rr$ for all $d \geq 0$, and
\item if $\omega_l = 0$, then there exists $d > 0$ such that $\scrB_d(S)$ is an infinite dimensional vector space over $\rr$.
\end{enumerate} 
\end{enumerate}
\end{cor}

\begin{example}\label{exex}
Consider a sequence of key-forms starting with $x,y,y^2 - x^5$ (which corresponds to choices $\omega_1 := 5/2$ and $c_1 := 1$). 
\begin{enumerate}
\item Take $\omega_2 := 1$, $l := 2$, $c_{2,1} :=0$ and $c_{2,2} := 1$. Let $S$ be the region defined by $x \geq 1$, $y \geq 0$, $x \geq y^2 - x^5 \geq 0$. Then $\scrB_0(S) = \rr$ and $\scrB(S)$ is a a finitely generated algebra over $\rr$.
\item Take $\omega_2 := 0$, $l := 2$, $c_{2,1} :=0$ and $c_{2,2} := 1$. Let $S$ be the region defined by $x \geq 1$, $y \geq 0$, $1 \geq y^2 - x^5 \geq 0$. Then $\scrB_0(S) \supsetneq \rr$ and $\scrB(S)$ is a finitely generated algebra over $\rr$. 
\item Take $\omega_2 := 3/2$, $c_2 := 1$, $\omega_3 := 1$, $l := 3$, $c_{3,1} :=0$ and $c_{3,2} := 1$. Let $S$ be the region defined by $x \geq 1$, $y \geq 0$, $x \geq y^2 - x^5 - yx^{-1} \geq 0$. Then $\scrB_0(S) = \rr$ and $\scrB(S)$ is not a finitely generated algebra over $\rr$. But $\scrB_d(S)$ is finite dimensional over $\rr$ for all $d$.
\item Take $\omega_2 := 3/2$, $c_2 := 1$, $\omega_3 := 0$, $l := 3$, $c_{3,1} :=0$ and $c_{3,2} := 1$. Let $S$ be the region defined by $x \geq 1$, $y \geq 0$, $1 \geq y^2 - x^5 - yx^{-1} \geq 0$. Then $\scrB_0(S) = \rr$ and $\scrB(S)$ is not a finitely generated algebra over $\rr$. Moreover, there exists $d > 0$ such that $\scrB_d(S)$ is infinite dimensional over $\rr$.
\end{enumerate}
\end{example}

\subsubsection{Two tentacles which behave like one}
In this subsection we assume that
\begin{enumerate}
\item $\xi$ is an indeterminate,
\item $\phi(x) = \sum_{j=1}^k a_j x^{m_j/2} \in \rr[x^{1/2}, x^{-1/2}]$ for some positive integer $k$ and integers $m_1 > m_2 > \cdots > m_k$ such that $\gcd(m_1, \ldots, m_k) = 1$, and
\item $\omega \in \Q$, $\omega < m_k/2 = \ord_x(\phi)$.
\end{enumerate}
Define 
\begin{align*}
\phi_1(x) &:= \sum_{j=1}^k a_j x^{m_j} + \xi x^{2\omega}, \\
\phi_2(x) &:= \sum_{j=1}^k (-1)^{m_j} a_j x^{m_j} + \xi x^{2\omega}.
\end{align*}

\begin{prop} \label{almost-single}
Let $\tilde S, S_1, S_2$ be semialgebraic subsets of $\rr^2$ which satisfy assumptions \ref{single-assumption} and \ref{regular-assumption}. Assume that the generic degree-wise Puiseux series (see Definition \ref{deg-wise-puiseux}) corresponding to $\tilde S, S_1,S_2$ is respectively $\phi(x) + \xi x^\omega$, $\phi_1(x) + \xi x^{2\omega}$, and $\phi_2(x) + \xi x^{2\omega}$. Set $S=S_1\cup S_2$. Then $\scrB(S)$ is integral over $\scrB(\tilde S)$. In particular, for each of the Questions \ref{one}--\ref{four}, its answer is positive for $S$ iff it is positive for $\tilde S$.
\end{prop}

\begin{proof}
Consider the map $(x,y) \mapsto (x^2,y)$. Then it is not hard to see that $\delta^*_{S_1}$ and $\delta^*_{S_2}$ extend $\delta^*_{\tilde S}$ via the pull-back by $f$. \cite[Lemma A.3]{subalsection} then shows that $\scrB(S)$ is the integral closure of $f^*\scrB(\tilde S)$. The proposition follows immediately.
\end{proof}

\begin{example}[The example of Section \ref{count} revisited]
The generic degree-wise Puiseux series corresponding to $S_1,S_2$ of Section \ref{count} are respectively
\begin{align*}
\tilde \phi_1(x,\xi) &:= - x^3 + x^{-2} + \xi x^{-3},\\
\tilde \phi_2(x,\xi) &:= x^3 + x^{-2} + \xi x^{-3}.
\end{align*}  
Let $\tilde S$ be a tentacle with generic degree-wise Puiseux series $\tilde \phi(x, \xi) := x^{3/2} + x^{-1} + \xi x^{-3/2}$; we may construct such $\tilde S$ using the procedure in Section \ref{construction}; e.g.\ take $\tilde S$ to be the set defined by $x \geq 1,\ y \geq 0,\ 1 \geq y^2 - x^3 - 2yx^{-1} \geq 0$. Then it follows exactly as in the last case of Example \ref{exex} that $\scrB_0(\tilde S) = \rr$ and there exists $d > 0$ such $\scrB_d(\tilde S)$ is infinite dimensional over $\rr$. Proposition \ref{almost-single} therefore implies that the same is true for $S := S_1 \cup S_2$, as it was indeed shown in Section \ref{count}.  
\end{example}
\section{Bounded polynomials}\label{bound}

In this last section we show how to interpret the algebra $\B(S)$ as the algebra $\B_0(S')$ of another set $S'$. In this way, we can produce examples and  counterexamples to Question \ref{three} from counterexamples to Question \ref{one}.

\begin{proposition}
Let $S\subseteq \R^n$ be a set. Define $$S':=\left\{ (a,s)\in \R^{n+1}\mid a\in S, \Vert a\Vert \geq 1, \Vert a\Vert^2s^2\leq 1\right\}.$$ Then $\B(S) \cong \B_0(S')$ (via the identification $p\in \B_d(S) \leftrightarrow p\cdot t^d $, where $t$ is the last coordinate function on $\rr^{n+1}$). 
\end{proposition}
\begin{proof} We can assume that $S$ contains only points with $\Vert a\Vert\geq 1.$
For ``$\subseteq$'' we start with $p\in \B_d(S)$. This implies $\vert p\vert \leq D\Vert x\Vert^d$ on $S$, for some large enough $D$. So for $(a,s)\in S'$ we have $$\vert p(a)s^d\vert \leq \vert p(a)\vert \frac{1}{\Vert a\Vert^d}\leq D.$$ So $p\cdot t^d\in \B_0(S'),$ which proves the claim. For ``$\supseteq$'' take $q=\sum_{i=0}^dp_i(x)t^i\in \B_0(S').$ There is some $C$ such that $$\left\vert\sum_{i=0}^d \frac{p_i(a)}{\Vert a\Vert^i}r^i \right\vert=\left\vert q\left(a,\frac{r}{\Vert a\Vert}\right)\right\vert\leq C$$ for all $a\in S, r\in[0,1].$ From Lemma \ref{help} below it follows that there is some $D$ such that $$\left\vert\frac{p_i(a)}{\Vert a\Vert^i}\right\vert\leq D$$ for all $a\in S,$ and this implies $p_i\in\B_i(S),$ and thus $q\in \B(S).$
\end{proof}
\begin{lemma}\label{help}
	If  a univariate polynomial $p\in \R[t]$ fulfills $\vert p\vert \leq C$ on $[0,1]$, then the size of the coefficients of $p$ is bounded by a constant depending only on $\deg(p)$ and $C.$\end{lemma}
\begin{proof}
Write $p=\sum_{i=0}^d p_it^i.$ We have $$-C\leq \sum_i p_i \left(\frac{1}{n}\right)^i\leq C$$ for all $n\in \N$. So the coefficient tuple $(p_0,\ldots,p_d)$ of $p$ lies in a polytope whose normal vectors are  $\pm\left(1,\frac1n, \frac{1}{n^2},\ldots,\frac{1}{n^d}\right),$ for $n=1,\ldots,d+1$. Since these vectors are linearly independent, this polytope is compact.\end{proof}

So from examples where $\B(S)$ is not finitely generated, we obtain new examples for which $\B_0(S')$ is not finitely generated. For example, we get  new regular examples in dimension three from the examples in Section \ref{count} and Section \ref{plane}.


\begin{bibdiv}
\begin{biblist}



\bib{amoh}{article}{
    AUTHOR = {S. S. Abhyankar and T. T. Moh},
     TITLE = {Embeddings of the line in the plane},
   JOURNAL = {J. Reine Angew. Math.},
    VOLUME = {276},
      YEAR = {1975},
     PAGES = {148--166},
}

\bib{cimane}{article}{
 AUTHOR = {J. Cimpri\v{c} and M. Marshall and T. Netzer},
 TITLE={Closures of quadratic modules},
 JOURNAL={Israel J. Math},
 YEAR={to appear},
}

\bib{friedland-milnor}{article}{
    AUTHOR = {S. Friedland and J. Milnor},
     TITLE = {Dynamical properties of plane polynomial automorphisms},
   JOURNAL = {Ergodic Theory Dynam. Systems},
    VOLUME = {9},
      YEAR = {1989},
}


\bib{havi}{article}{
AUTHOR={E. K. Haviland},
TITLE={On the moment problem for distribution
functions in more than one dimension II}, 
JOURNAL ={Amer. J. Math.},
VOLUME={58},
YEAR={1936},
PAGES={164--168},
}


\bib{kr}{article}{
AUTHOR={S. Krug},
TITLE={Geometric interpretations of a counterexample to HilbertÕs 14th problem and rings of bounded polynomials on semialgebraic sets},
JOURNAL={Preprint},
YEAR={2011},
}

\bib{kuma}{article}{
AUTHOR={S. Kuhlmann and M. Marshall},
TITLE={Positivity, sums of squares and the multi-dimensional moment problem},
JOURNAL={Trams. Amer. Math. Soc.},
VOLUME={354},
NUMBER={11},
YEAR={2002},
PAGES={4285--4301},
}

\bib{lazarsfeld}{article}{
	AUTHOR ={R. Lazarsfeld},
	TITLE  = {Lectures on linear series},
BOOKTITLE  = {Complex algebraic geometry ({P}ark {C}ity, {UT}, 1993)},
    VOLUME = {3},
     PAGES = {161--219},
      YEAR = {1997},
}

\bib{maclane-key}{article}{
    AUTHOR = {S. MacLane},
     TITLE = {A construction for absolute values in polynomial rings},
   JOURNAL = {Trans. Amer. Math. Soc.},
    VOLUME = {40},
      YEAR = {1936},
    NUMBER = {3},
     PAGES = {363--395},
  }


\bib{ma}{book}{,
    AUTHOR = {M. Marshall},
     TITLE = {Positive polynomials and sums of squares},
    SERIES = {Mathematical Surveys and Monographs},
    VOLUME = {146},
 PUBLISHER = {American Mathematical Society},
   ADDRESS = {Providence, RI},
      YEAR = {2008},
     PAGES = {xii+187},
 }



\bib{contractibility}{article}{
AUTHOR={P. Mondal},
TITLE={An effective criterion for algebraic contractibility of rational curves},
JOURNAL={Preprint},
YEAR={2013},
}


\bib{non-negative-valuation}{article}{
AUTHOR={P. Mondal},
TITLE={How to determine the \em sign of a valuation on $\C[x,y]$? },
JOURNAL={Preprint},
YEAR={2013},
}

\bib{subalsection}{article}{
AUTHOR={P. Mondal},
TITLE={Is the intersection of two finitely generated subalgebras of a polynomial ring also finitely generated?},
JOURNAL={Preprint},
YEAR={2013},
}

\bib{sub1}{article}{
AUTHOR={P. Mondal},
TITLE={Projective completions of affine varieties via degree-like functions?},
JOURNAL={Preprint},
YEAR={2010},
}

\bib{ne}{article}{
AUTHOR={T. Netzer},
TITLE={Stability of quadratic modules},
JOURNAL={Manuscr. math.},
VOLUME={129},
NUMBER={2},
PAGES={251--271},
YEAR={2009},
} 

\bib{ne2}{article}{
AUTHOR={T. Netzer},
TITLE={An elementary proof of Schm\"udgen's theorem on the moment problem of closed semialgebraic sets},
JOURNAL={Proc.  Amer. Math. Soc.},
VOLUME={136},
PAGES={529--537},
YEAR={2008},
} 


\bib{pl}{book}{
AUTHOR={D. Plaumann},
TITLE={Bounded polynomials, sums of squares and the moment problem},
SERIES={Phd-thesis, University of Konstanz},
PUBLISHER={Konstanz Online Publikation System},
YEAR={2008},
}

\bib{plsc}{article}{
AUTHOR={D. Plaumann and C. Scheiderer},
TITLE={The ring of bounded polynomials on a semi-algebraic set},
JOURNAL={Trans. Am. Math. Soc.},
VOLUME={364},
PAGES={4663--4682},
YEAR={2012},
}


\bib{posc}{article}{
AUTHOR={V. Powers and C. Scheiderer},
TITLE={The moment problem for non-compact semialgebraic sets},
JOURNAL={Adv. Geom.},
VOLUME={1}, 
PAGES={71--88},
YEAR={2001},
}

\bib{sc}{article}{
AUTHOR={C. Scheiderer},
TITLE={Non-existence of degree bounds for weighted sums of squares representations},
JOURNAL={J. Complexity},
VOLUME={21},
PAGES={823--844},
YEAR={2005},
}

\bib{sch1}{article}{
AUTHOR={K. Schm\"udgen},
TITLE={The $K$-moment problem for compact
semi-algebraic sets}, 
JOURNAL={Math. Ann.},
VOLUME={289},
YEAR={1991}, 
PAGES={203--206},
}

\bib{sch2}{article}{
AUTHOR={K. Schm\"udgen},
TITLE={On the moment problem of closed
semi-algebraic sets}, 
JOURNAL={J. reine angew. Math.},
VOLUME={558},
YEAR={2003}, 
PAGES={225--234},
}

\bib{suzuki}{article}{
    AUTHOR = {M. Suzuki},
     TITLE = {Propri\'et\'es topologiques des polyn\^omes de deux variables
              complexes, et automorphismes alg\'ebriques de l'espace {${\bf
              C}^{2}$}},
   JOURNAL = {J. Math. Soc. Japan},
  FJOURNAL = {Journal of the Mathematical Society of Japan},
    VOLUME = {26},
      YEAR = {1974},
     PAGES = {241--257},
}

\bib{szpiro}{article}{,
    AUTHOR = {L. Szpiro},
     TITLE = {Fonctions d'ordre et valuations},
   JOURNAL = {Bull. Soc. Math. France},
  FJOURNAL = {Bulletin de la Soci\'et\'e Math\'ematique de France},
    VOLUME = {94},
      YEAR = {1966},
     PAGES = {301--311},
 
}


\bib{vi}{article}{
AUTHOR={C. Vinzant},
TITLE={Real radical initial ideals}, 
JOURNAL={Journal of Algebra},
VOLUME={352},
NUMBER={1},
YEAR={2012}, 
PAGES={392Ð-407}, }

\end{biblist}
\end{bibdiv}
\end{document}